\documentclass[12pt,leqno]{article}
\usepackage{amsfonts,amsthm,amsmath}

\theoremstyle{plain}
\newtheorem{thm}{Theorem}[section]

\newtheorem{lem}{Lemma}[section]

\newtheorem{conj}{Conjecture}[section]

\theoremstyle{definition}
\newtheorem{df}{Definition}[section]
\newtheorem{rem}{Remark}[section]

\newcommand{\FF}{\mathbb{F}}
\newcommand{\ZZ}{\mathbb{Z}}

\newcommand{\CC}{\mathbb{C}}

\newcommand{\QQ}{\mathbb{Q}}

\newcommand{\NN}{\mathbb{N}}

\newcommand{\HH}{\mathbb{H}}

\DeclareMathOperator{\Mat}{Mat}
\DeclareMathOperator{\Ima}{Im}

\DeclareMathOperator{\diag}{diag}

\begin{document}

\title{{
On Eisenstein polynomials and zeta polynomials
}
\footnote{This work was supported by JSPS KAKENHI (18K03217).}
}

\author{
Tsuyoshi Miezaki
\thanks{Faculty of Education, University of the Ryukyus, Okinawa  
903-0213, Japan 
miezaki@edu.u-ryukyu.ac.jp
(Corresponding author)
}
}

\date{}
\maketitle

\begin{abstract}
Eisenstein polynomials, which were defined by 
Oura, are analogues of the 
concept of an Eisenstein series. 
Oura conjectured that there exist some analogous properties 
between Eisenstein series and Eisenstein polynomials.
In this paper, 
we provide new analogous properties of Eisenstein polynomials and 
zeta polynomials. 
These properties are finite analogies of 
certain properties of Eisenstein series. 
\end{abstract}

{\small
\vspace{-0.15in}

\quad
}


\section{Introduction}
In the present paper, 
we discuss some analogies between 
Eisenstein series, Eisenstein polynomials, and zeta polynomials. 
First we define Eisenstein series and Eisenstein polynomials. 
For $g\in \NN$, let 
\[
\Gamma_g:=\{M\in \Mat(2g,\ZZ)\mid {}^t M J_gM=J_g\}, 
\]
where 
$J_g=
\begin{pmatrix}
{\bf 0} & {\bf 1_g}\\
-{\bf 1_g} & {\bf 0}
\end{pmatrix}, 
$
and ${\bf 1_g}$ is the identity matrix of degree $g$. 
Let $\HH_g$ be the Siegel upper half plane, namely, 
\[
\HH_g:=\{
M\in \Mat(g,\ZZ)\mid {}^t M=M, \Ima M>0
\}. 
\]
Let $f$ be a holomorphic function on $\HH_g$. 
Then $f$ is called a Siegel modular form
for $\Gamma_g$ of weight $k$ if $f$ satisfies
\[
f(MZ)= \det( CZ + D) ^k f(Z) \mbox{ for } 
M= \begin{pmatrix}
A&B\\
C&D
\end{pmatrix}\in \Gamma_g
\]
and if $f$ is also holomorphic at cusps. 
We write $M(\Gamma_g)$ for the ring of the Siegel modular forms. 
The Siegel modular forms are 
considered to be $\Gamma_g$-invariant 
functions 
(see \cite{F1,F2,K} for details about Siegel modular forms).

Next, we introduce 
some typical examples of 
Siegel modular forms. 
For $g\in \NN$, let 
\[
\Delta_{g,0}:=
\left\{
\begin{pmatrix}
\ast&\ast\\
{\bf 0_n}&\ast
\end{pmatrix}
\in \Gamma_g
\right\}, 
\] 
where 
${\bf 0_g}$ is the zero matrix of degree $g$. 
The Siegel Eisenstein series is defined as follows: 
\[
\psi_k^{\Gamma_g}(Z)= \sum_
{
\begin{pmatrix}
A&B\\
C&D
\end{pmatrix}
:\Delta_g\backslash \Gamma_g
} \det( CZ + D) ^{-k}, 
\]
for even $k>g+1$, where the summation is 
over a full set of representatives for the coset $\Delta_g\backslash \Gamma_g$. 

In the following, 
we define an Eisenstein polynomial. 
Let 
\[
H_g:=\left\langle 
\left(\frac{1+\sqrt{-1}}{2}\right)^g 
\left((-1)^{({\bf a},{\bf b})}\right)_{{\bf a},{\bf b}\in \FF_2^g}, 
\diag\left(\sqrt{-1})^{{}^t{\bf a}S{\bf a}};{\bf a}\in \FF_2^g\right)
\right\rangle. 
\]
Then,
$H_g$ acts on the space $\CC[x_{\bf a}:{\bf a}\in \FF_2^g]$ 
in the natural way and 
we define the $H_g$-invariant subspace of $\CC[x_{\bf a}:{\bf a}\in \FF_2^g]$ 
as follows: 
\begin{align*}
&\CC[x_{\bf a}:{\bf a}\in \FF_2^g]^{H_g}\\
&:=
\{
f(x_{\bf a}:{\bf a}\in \FF_2^g)\in \CC[x_{\bf a}:{\bf a}\in \FF_2^g]\\
&\hspace{30pt}\mid 
f(M {}^t(x_{\bf a}:{\bf a}\in \FF_2^g))=f(x_{\bf a}:{\bf a}\in \FF_2^g), M\in H_g
\}. 
\end{align*}

Here is a typical example of 
$\CC[x_{\bf a}:{\bf a}\in \FF_2^g]^{H_g}$. 
Oura defined an
Eisenstein polynomial 
as follows: 
\[
\varphi_\ell^{H_g}(x_{\bf a}:{\bf a}\in \FF_2^g)
=\frac{1}{|H_g|}\sum_{\sigma\in H_g}(\sigma x_{\bf 0})^\ell 
\]
\cite{Oura1,Oura2}.
It is straightforward to show that the 
Eisenstein polynomial is in $\CC[x_{\bf a}:{\bf a}\in \FF_2^g]^{H_g}$. 

Here, 
we introduce an expression relating 
$\CC[x_{\bf a}:{\bf a}\in \FF_2^g]^{H_g}$ and 
$M(\Gamma_g)$. 
For $f \in \CC[x_{\bf a}:{\bf a}\in \FF_2^g]^{H_g}$, 
we construct the elements of $\Gamma_g$ as follows: 
\begin{align*}
Th: \CC[x_{\bf a}:{\bf a}\in \FF_2^g]^{H_g} &\rightarrow 
M(\Gamma_g)\\
x_{\bf a}&\mapsto f_{\bf a}(\tau)=\sum_{{\bf b}\in \ZZ^g, {\bf a}\equiv {\bf b}\pmod{2}}
\exp(\pi i {}^t{\bf b}\tau {\bf b}/2). 
\end{align*}
The map $Th$ is called the theta map. 

The elements of both $M(\Gamma_g)$ and 
$\CC[x_{\bf a}:{\bf a}\in \FF_2^g]^{H_g}$ are ``invariant functions" 
and the Eisenstein series and the Eisenstein polynomial are 
``average functions" of the groups. Therefore, 
these two objects are expected to have similar properties. 
Moreover, for $f \in \CC[x_{\bf a}:{\bf a}\in \FF_2^g]^{H_g}$, 
it is expected that $f$ and $Th(f)$ have similar properties. 

Table \ref{Tab:sum} shows a summary of the concepts that we have introduced thus far. 
\begin{table}[thb]
\caption{Summary of our objects}
\label{Tab:sum}
\begin{center}
\begin{tabular}{c|c}
\noalign{\hrule height0.8pt}
$\Gamma_g$ & $H_g$\\\hline
$M(\Gamma_g)$ & $\CC[x_{\bf a}:{\bf a}\in \FF_2^g]^{H_g}$\\\hline
Eisenstein series & Eisenstein polynomials\\\hline
$f$ & $Th(f)$\\
\noalign{\hrule height0.8pt}
\end{tabular}
\end{center}
\end{table}

In the following, 
we consider the case $g=1$. 
The explicit generators of $H_1$ are written as follows: 
\[
H_1=
\left\langle
\frac{1}{2}
\begin{pmatrix}
1+\sqrt{-1}&1+\sqrt{-1}\\
1+\sqrt{-1}&1-\sqrt{-1}
\end{pmatrix}, 
\begin{pmatrix}
1&0\\
0&\sqrt{-1}
\end{pmatrix}
\right\rangle. 
\]
Then the Eisenstein polynomial $\varphi_{\ell}^{H_1}(x_0,x_1)$ is written as follows: 
\[
\varphi_{\ell}^{H_1}(x_0,x_1)=\frac{1}{|H_1|}\sum_{\sigma\in H_1}(\sigma x_0)^{\ell}. 
\]
It is known that the ring 
$\CC[x_0,x_{1}]^{H_1}$ is generated by two elements \cite{ST}: 
\[
\CC[x_0,x_{1}]^{H_1}=\langle \varphi_{8}^{H_1}(x_0,x_1),\varphi_{12}^{H_1}(x_0,x_1)\rangle. 
\]
Therefore, for $\ell\not\equiv 0\pmod{4}$ and $\ell=4$, 
$\varphi_{\ell}^{H_1}(x_0,x_1)\equiv 0$. 
For $\varphi_{\ell}^{H_1}(x_0,x_1)\not\equiv 0$, 
we denote by $\widetilde{\varphi_{\ell}^{H_1}}(x_0,x_1)$ 
the polynomial $\varphi_{\ell}^{H_1}(x_0,x_1)$ divided by 
its $x_0^\ell$ coefficient.
We give some examples. 
\begin{table}[thb]
\label{Tab:sum}
\begin{center}
\begin{tabular}{c|c}
\noalign{\hrule height0.8pt}
$\ell$ & $\widetilde{\varphi_{\ell}^{H_1}}(x_0,x_1)$\\\hline
$8$ & $x_0^8+14 x_0^4 x_1^4+x_2^8$\\\hline
$12$ & $x_0^{12}-33 x_0^8 x_1^4-33 x_0^4 x_1^8+x_1^{12}$\\
\noalign{\hrule height0.8pt}
\end{tabular}
\end{center}
\end{table}






In \cite{Ouratalk,Oura3}, several analogies between 
Eisenstein series and 
Eisenstein polynomials were reported. 
Suppose $p$ is a prime number and $v_p$ is the corresponding 
value for the field $\QQ$. 
Then $a\in \QQ$ is called $p$-integral if $v_p(a)\geq 0$. 
Eisenstein series have the following properties: 
\begin{enumerate}
\item [(1)]
All of the zeros of the Eisenstein series are 
on the circle 
$\{e^{\sqrt{-1}\theta}\mid \pi/2\leq \theta \leq 2\pi/3\}$ 
\cite{RS}. 
\item [(2)]
The zeros of the Eisenstein series $\psi_k^{\Gamma_1}(z)$ are the same as those for $\psi_{k+2}^{\Gamma_1}(z)$ 
\cite{Nozaki}. 
\item [(3)]
For odd prime $p$, where $p\geq 5$, 
the coefficients of the Eisenstein series $\psi_{p-1}^{\Gamma_1}(z)$ 
are $p$-integral 
\cite[P.~233, Theorem 3]{IR}, \cite{Kob}. 
\end{enumerate}
Oura's conjecture states that 
the analogous properties of $(1),(2)$, and $(3)$ also hold for 
$Th(\widetilde{\varphi_\ell})$. 
Namely, 
\begin{conj}[\cite{Ouratalk,Oura3}]\label{conj;Oura}
\begin{enumerate}
\item [{\rm (1)}]
All of the zeros of $Th(\widetilde{\varphi_\ell^{H_1}})$ are 
on the circle 
$\{e^{\sqrt{-1}\theta}\mid \pi/2\leq \theta \leq 2\pi/3\}$. 
\item [{\rm (2)}]
The zeros of $Th(\widetilde{\varphi_\ell^{H_1}})$ are the same as 
those of $Th(\widetilde{\varphi_{\ell+4}^{H_1}})$. 
\item [{\rm (3)}]
Let $p$ be an odd prime. 
The coefficients of $Th(\widetilde{\varphi_{2(p-1)}^{H_1}})$ 
are $p$-integral. 
\end{enumerate}
\end{conj}

To explain the above results, 
we introduce the zeta polynomials, 
which were defined by Duursma \cite{D1}.
Analogous to coding theory, we say $f \in \CC[x_0, x_1]$ 
is the formal weight enumerator of degree $n$ 
if $f$ is a homogeneous polynomial of degree $n$ 
and the coefficient of $x_0^n$ is one. 
Also, for
\[
f (x_0, x_1) = x_0^n +
\sum_{i=d}^n
A_ix_0^{n-i}x_1^i\ (A_d \neq 0),
\]
$d$ is the minimum distance of $f$. 
Let $R$ be a commutative ring and $R[[T]]$ be the formal power series 
ring over $R$. 
For
$Z(T) =
\sum_{i=0}^{\infty}a_nT_n \in R[[T]]$, 
$[T^k]Z(T)$ denotes the coefficient $a_k$. 
The following lemma follows: 
\begin{lem}[cf.~\cite{D1}]\label{lem:D}
Let $f$ be a formal weight enumerator of degree $n$, 
$d$ be the minimum distance, and $q$ be any real number not one.
Then there exists a unique polynomial $P_f(T) \in \CC[T]$ 
of degree at most $n-d$ such that the following equation holds: 
\[
[T^{n-d}]
\frac{P_f(T)}
{(1 - T )(1 -qT ) }
(x_0T + x_1(1 - T ))^n =
\frac{f (x_0, x_1) - x_0^n}
{q - 1}.
\]
\end{lem}

\begin{df}[cf.~\cite{D2}]
For a formal weight enumerator $f$, we call the polynomial 
$P_f(T)$ determined in Lemma \ref{lem:D} the zeta polynomial of $f$ with respect to $q$. 
If all the zeros of $P_f(T)$ have absolute value $1/\sqrt{q}$, 
then $f$ satisfies the Riemann hypothesis analogues (RHA). 
\end{df}

We investigate the zeta polynomials of the Eisenstein polynomials for $q=2$. 
In the following, we assume that $q=2$. 
Below are the cases of $\ell = 8$ and $\ell = 12$: 
\begin{table}[thb]
\label{Tab:sum}
\begin{center}
\begin{tabular}{c|c}
\noalign{\hrule height0.8pt}
$\ell$ & $P_{\widetilde{\varphi_{\ell}^{H_1}}}(T)$\\\hline
$8$ & $\frac{1}{5}+\frac{2 T}{5}+\frac{2 T^2}{5}$\\\hline
$12$ & $-\frac{1}{15}-\frac{2 T}{15}-\frac{2 T^2}{15}+\frac{4 T^4}{15}+\frac{8 T^5}{15}+\frac{8 T^6}{15}$\\
\noalign{\hrule height0.8pt}
\end{tabular}
\end{center}
\end{table}






The main purpose of the present paper 
is to show that 
Oura's observation for the zeta polynomial associated 
with Eisenstein polynomials holds: 
\begin{thm}\label{thm:main}
\begin{itemize}
\item [{\rm (I)}]
\begin{enumerate}
\item [{\rm (1)}]
$P_{\widetilde{\varphi_{\ell}^{H_1}}}(T)$ satisfies RHA. 

\item [{\rm (2)}]
The zeros of $P_{\widetilde{\varphi_{\ell}^{H_1}}}(T)$ interlace  those of $P_{\widetilde{\varphi_{\ell+4}^{H_1}}}(T)$. 

\item [{\rm (3)}]
Let $p$ be an odd prime with $p\neq 5$. Then 
the coefficients of $P_{\widetilde{\varphi_{2(p-1)}^{H_1}}}(T)$ are $p$-integral. 

\end{enumerate}
\item [{\rm (II)}]
Let $p$ be an odd prime. Then 
the coefficients of $\widetilde{\varphi_{2(p-1)}^{H_1}}(x_0,x_1)$ are $p$-integral. 

\item [{\rm (III)}]
Conjecture \ref{conj;Oura} (3) is true. 
\end{itemize}
\end{thm}

In Section $2$, 
the proof of Theorem \ref{thm:main} is provided along with concluding remarks. 

\section{Proof of Theorem \ref{thm:main}}

In this section, 
we provide the proof of Theorem \ref{thm:main}. 

\subsection{Preliminaries}

Before proving Theorem \ref{thm:main}, 
we first review some properties of Eisenstein polynomials and 
zeta polynomials. 

The explicit form of the Eisenstein polynomials $\varphi_\ell^{H_1}(x_0,x_1)$ are given by 
\begin{thm}[cf.~\cite{T}]\label{thm:T}
\[
\varphi_\ell^{H_1}(x_0,x_1)=((-1)^{\ell/4}+2^{(\ell-4)/2})(x_0^\ell+x_1^\ell)
+\sum_{0<j<\ell,j\equiv 0\pmod{4}}(-1)^{\ell/4}\binom{\ell}{j}x_0^{\ell-j}x_1^j. 
\]
\end{thm}

The zeta polynomial $P_f(T)$ associated with $f$ is related 
to the normalized weight enumerator of $f$ as follows: \begin{df}[cf.~\cite{D3}]\label{df:nwe}
For a formal weight enumerator $f(x_0,x_1)=\sum_{i=0}^{n} A_ix_0^{n-i}x_1^i$, 
we define the normalized weight enumerator as follows:
\begin{align*}
N_f(t)=\frac{1}{q-1}\sum_{i=d}^{n} A_i/\binom{n}{i}t^{i-d}. 
\end{align*}
\end{df}
$P_f(T)$ and $N_f(t)$ have the following relation: 
\begin{thm}[cf.~\cite{D3}]\label{thm:D3}
For a given formal weight enumerator $f(x,y)$ with minimum distance $d$, 
the zeta polynomial $P_f(T)$ and the normalized 
weight enumerator $N_f(t)$ have the following relation: 
\[
\frac{P_f(T)}{(1-T)(1-qT)}(1-T)^{d+1}\equiv N_f\left(\frac{T}{1-T}\right)\pmod{T^{n-d+1}}. 
\]
\end{thm}

To prove Theorem \ref{thm:main}, 
we provide the explicit formula of the zeta function associated with 
Eisenstein polynomials: 
\begin{thm}\label{thm:zeta}
The zeta polynomial associated with
Eisenstein polynomials $\widetilde{\varphi_{\ell}^{H_1}}$ is written as 
follows: 
\begin{align*}
P_{\widetilde{\varphi_{\ell}^{H_1}}}(T)=\frac{1}{(-1)^{\ell/4}+2^{(\ell-4)/2}}\frac{(-1)^{\ell/4}+2^{(\ell-4)/2}T^{\ell-4}}{1-2T+2T^2}. 
\end{align*}
\end{thm}

\begin{proof}
Let 
$N_{\widetilde{\varphi_{\ell}^{H_1}}}$ be the normalized weight enumerator of 
$\varphi_{\ell}^{H_1}$. 
By Definition \ref{df:nwe}, 
we have 
\begin{align*}
N_{\widetilde{\varphi_\ell^{H_1}}}(t)
&=
\sum_{0<j<\ell,j\equiv 0\pmod{4}}\frac{(-1)^{\frac{\ell}{4}}}{(-1)^{\ell/4}+2^{(\ell-4)/2}}
t^{j - 4}+t^{\ell- 4}\\
&\equiv 
\frac{(-1)^{\ell/4}}{(-1)^{\ell/4}+2^{(\ell-4)/2}}\frac{1}{1-t^4}
+\left(1-\frac{(-1)^{\ell/4}}{(-1)^{\ell/4}+2^{(\ell-4)/2}}\right)t^{\ell- 4}\pmod{t^{\ell-3}}. 
\end{align*}
Then, by Theorem \ref{thm:D3}, 
we have 
\begin{align*}
&\frac{P_{\widetilde{\varphi_\ell^{H_1}}}(T)}{(1-T)(1-2T)}(1-T)^{5}\equiv 
N_{\widetilde{\varphi_\ell^{H_1}}}\left(\frac{T}{1-T}\right)\pmod{T^{\ell-3}}\\
\Leftrightarrow &
P_{\widetilde{\varphi_\ell^{H_1}}}(T)
\equiv 
N_{\widetilde{\varphi_\ell^{H_1}}}\left(\frac{T}{1-T}\right)
\frac{(1-T)(1-2T)}{(1-T)^{5}} \pmod{T^{\ell-3}}\\
&
\equiv 
\frac{(-1)^{\ell/4}}{(-1)^{\ell/4}+2^{(\ell-4)/2}}\frac{(1-T)^4}{(1-T)^4-T^4}\frac{(1-T)(1-2T)}{(1-T)^{5}}\\
&+\left(1-\frac{(-1)^{\ell/4}}{(-1)^{\ell/4}+2^{(\ell-4)/2}}\right)\left(\frac{T}{1-T}\right)^{\ell- 4}
\frac{(1-T)(1-2T)}{(1-T)^{5}} \pmod{T^{\ell-3}}\\
&
\equiv 
\frac{(-1)^{\ell/4}}{(-1)^{\ell/4}+2^{(\ell-4)/2}}\frac{1}{1-2T+2T^2}\\
&+\left(\frac{2^{(\ell-4)/2}}{(-1)^{\ell/4}+2^{(\ell-4)/2}}\right)\frac{T^{\ell-4}}{(1-T)^{\ell}}
 \pmod{T^{\ell-3}}\\
&\equiv \frac{1}{(-1)^{\ell/4}+2^{(\ell-4)/2}}
\frac{(-1)^{\ell/4}+2^{(\ell-4)/2}T^{\ell-4}}{1-2T+2T^2} \pmod{T^{\ell-3}}. 
\end{align*}
Then, we have 
$$
P_{\widetilde{\varphi_{\ell}^{H_1}}}(T)=\frac{1}{(-1)^{\ell/4}+2^{(\ell-4)/2}}\frac{(-1)^{\ell/4}+2^{(\ell-4)/2}T^{\ell-4}}{1-2T+2T^2}. 
$$

\end{proof}

\subsection{Proof of Theorem \ref{thm:main}}
In this section, we provide the proof of Theorem \ref{thm:main}. 

First, the following lemma: 
\begin{lem}\label{lem:mod}
For $\ell=2(p-1)$ for some odd prime $p$ with $p\neq 5$, 
\[
(-1)^{\ell/4}+2^{(\ell-4)/2}\not\equiv 0\pmod{p}. 
\]
\end{lem}
\begin{proof}
For $\ell=2(p-1)$ for some odd prime $p$ with $p\neq 5$, 
\[
(-1)^{\ell/4}+2^{(\ell-4)/2}\not\equiv 0\pmod{p}\cdots (\ast). 
\]
(Note that $\varphi_4^{H_1}(x_0,x_1)\equiv 0$. 
Therefore, we exclude the case $p=3$.)
We consider the following two cases.
\begin{enumerate}
\item [(1)]
Let $p=4n-1$ for some $n\in \NN$. 
Then $(-1)^{\ell/4}+2^{(\ell-4)/2}
=-1+2^{p-3}$. By Fermat's little theorem,
$-1+2^{p-3}\equiv -3\times 2^{p-3}\not\equiv 0\pmod{p}$. 
\item [(2)]
Let $p=4n+1$ for some $n\in \NN$. 
Then $(-1)^{\ell/4}+2^{(\ell-4)/2}
=1+2^{p-3}$. By Fermat's little theorem,
$1+2^{p-3}\equiv 5\times 2^{p-3}\not\equiv 0\pmod{p}$. 
\end{enumerate}
\end{proof}

We now present the proof of Theorem \ref{thm:main}. 
\begin{proof}[Proof of Theorem \ref{thm:main}]
Clearly (I)--(1) and (I)--(2) follow from Theorem \ref{thm:zeta}. 

For (I)--(3), 
we recall that 
\[
P_{\widetilde{\varphi_{\ell}^{H_1}}}(T)=\frac{1}{(-1)^{\ell/4}+2^{(\ell-4)/2}}\frac{(-1)^{\ell/4}+2^{(\ell-4)/2}T^{\ell-4}}{1-2T+2T^2}. 
\]
Let $\ell=4m$,
if $m$ is even, 
then 
\begin{align*}
P_{\widetilde{\varphi_{\ell}^{H_1}}}(T)
&=\frac{1}{(-1)^{\ell/4}+2^{(\ell-4)/2}}\\
&\sum_{i=1}^{\ell/4-1}
(-1)^{i-1}
(4^{i-1}T^{4(i-1)}+2\times 4^{i-1}T^{4(i-1)+1}
+2\times 4^{i-1}T^{4(i-1)+2}). 
\end{align*}
If $m$ is odd, 
then 
\begin{align*}
P_{\widetilde{\varphi_{\ell}^{H_1}}}(T)
&=\frac{1}{(-1)^{\ell/4}+2^{(\ell-4)/2}}\\
&\sum_{i=1}^{\ell/4-1}
(-1)^{i}(4^{i-1}T^{4(i-1)}+2\times 4^{i-1}T^{4(i-1)+1}+2\times 4^{i-1}T^{4(i-1)+2}). 
\end{align*}

Then, from Lemma \ref{lem:mod}, 
the proof of Theorem \ref{thm:main} (I)--(3) is complete. 
(Note that $\varphi_4^{H_1}(x_0,x_1)\equiv 0$. 
Therefore, we exclude the case $p=3$.)

To show (II), we first recall that 
\[
\widetilde{\varphi_{\ell}^{H_1}}(x_0,x_1)=(x_0^\ell+x_1^\ell)+
\sum_{0<j<\ell,j\equiv 0\pmod{4}}
\frac{(-1)^{\ell/4}\binom{\ell}{j}}{((-1)^{\ell/4}+2^{\ell/2-2})}x_0^{\ell-j}x_1^j. 
\]
By Lemma \ref{lem:mod}, for $p\neq 5$ 
the coefficients of $\widetilde{\varphi_{2(p-1)}^{H_1}}(x_0,x_1)$ are $p$-integral.
For $p\neq 5$, $\widetilde{\varphi_{2(p-1)}^{H_1}}(x_0,x_1)=x_0^8+14x_0^4x_1^4+x_1^8$. Therefore, the coefficients of $\widetilde{\varphi_{8}^{H_1}}(x_0,x_1)$ are also $p$-integral.

Finally, we show (III). 
By Theorem \ref{thm:main} (II), 
the coefficients of 
\[
\widetilde{\varphi_{\ell}^{H_1}}(x_0,x_1)=(x_0^\ell+x_1^\ell)+
\sum_{0<j<\ell,j\equiv 0\pmod{4}}
\frac{(-1)^{\ell/4}\binom{\ell}{j}}{((-1)^{\ell/4}+2^{\ell/2-2})}x_0^{\ell-j}x_1^j
\]
are $p$-integral. For $g=1$, the theta map $f_0$ and $f_1$ have integral Fourier coefficients. This completes the proof. 
\end{proof}

\subsection{Concluding Remarks}
\begin{rem}
\begin{enumerate}
\item 
In the present paper, 
we only consider the genus one ($g=1$) case. 
For the cases with $g>1$, 
do the analogies still hold? 

\item 
The group $H_1$ is an example of a finite unitary reflection group. 
These groups are classified in \cite{ST}, 
which gives rise to a natural question: 
for the other unitary reflection groups, 
do our analogies still hold? 

\end{enumerate}
\end{rem}

\section*{Acknowledgments}
The authors thank Koji Chinen, Iwan Duursma, and Manabu Oura for their helpful discussions and contributions to this research.


\end{document}